\documentclass{amsart}

\title[Chern classes]{Chern classes and generators}
\author{Masaki  Kameko}
\address{
Department of Mathematical Sciences,
Shibaura Institute of Technology,
307 Minuma-ku Fukasaku, Saitama-City 337-8570, Japan}
 \email{kameko@shibaura-it.ac.jp}
\thanks{The author is partially supported by the Japan Society for the Promotion of Science, Grant-in-
Aid for Scientific Research (C) 22540102.}

\newtheorem{theorem}{Theorem}[section]
\newtheorem{proposition}{Proposition}[section]

\theoremstyle{definition}

\newtheorem{remark}{Remark}[section]

\usepackage[usenames]{color}

\newcommand{\spin}{\mathrm{Spin}}
\newcommand{\sq}{\mathrm{Sq}}

\begin{document}

\begin{abstract}
We give a simple proof for the fact that algebra generators of the mod 2 cohomology of classifying spaces of exceptional Lie groups are given by Chern classes and Stiefel-Whitney classes of certain representations.
\end{abstract}

\maketitle

\section{Introduction}

One of the standard tools to compute the mod $2$ cohomology of the classifying space $BG$  of a connected Lie group $G$ 
 is a spectral sequence converging to the mod $2$ cohomology of $BG$. In the case of $G=E_6$ or $E_7$, as an algebra over the Steenrod algebra, the $E_2$-term of the spectral sequence is generated by only two elements, one is of degree $4$ and the other is of degree $32$ or $64$, respectively. See \cite{kono-mimura-1975}, \cite{kono-mimura-shimada-1976}. To show the collapsing of the spectral sequence, it suffices to show that the algebra generator of degree $32$ or $64$ survives to the $E_\infty$-term, respectively. It could be done by showing that the algebra generator is represented by a characteristic class of some representation of $G$. It is also conjectured that the mod $2$ cohomology of $BE_8$ is also generated by these two elements as an algebra over the mod $2$ Steenrod algebra.
 
Kono, in \cite{kono-2005}, gives a simple proof for the following theorem by considering certain finite $2$-groups in exceptional Lie groups and computing Stiefel-Whitney classes of their representations.
 
\begin{theorem}\label{1}
For $G=F_4$, $E_6$, $E_7$, $E_8$, 
there exist representations $\rho_4$, $\rho_6$, $\rho_7$, $\rho_8$ such that
$w_{16}(\rho_4)$, $c_{16}(\rho_6)$, $c_{32}(\rho_7)$ and $w_{128}(\rho_8)$ are indecomposable in $H^{*}(BG;\mathbb{Z}/2)$. 
\end{theorem}

The case $G=E_6$ is proved by Kono and Mimura in \cite{kono-mimura-1975}, the case $G=E_7$ is also proved by Kono, Mimura and Shimada in \cite{kono-mimura-shimada-1976}.
The case $G=E_8$ is proved by Mimura and Nishimoto in \cite{mimura-nishimoto-2007} with rather complicated calculation.

In this paper, we give another simpler proof for this theorem using Chern classes of representations of spinor groups. The method in this paper is used by Schuster, Yagita and the author in \cite{schuster-yagita-2001} and \cite{kameko-yagita-2010},  where  the degree $4$ element in the integral cohomology of the above classifying spaces are studied in conjunction with Chern subring of classifying spaces.

\section{Chern classes}

Let us recall complex representation rings of spinor groups.
For $n\geq 6$, we consider $m=\left[ \dfrac{n}{2}\right]$, so that $n=2m$ or $2m+1$.
Let $T^m$ be a fixed maximal torus of  $\spin (n)$ and denote by $f_m:T^m \to \spin(n)$ the inclusion. We also denote by $f_1:T^1\to \spin(n)$ the composition of the inclusion of the maximal torus and the inclusion of the first factor $T^1$ into $T^m$. The complex representation ring of $\spin(n)$ is given by
\[
R(\spin (2m))=\mathbb{Z}[\lambda_1, \dots, \lambda_{m-2}, \Delta^{+}, \Delta^{-}], \]
and
\[
R(\spin (2m+1))=\mathbb{Z}[\lambda_1, \dots, \lambda_{m-1}, \Delta].
\]
The image $f_m^*(\lambda_i)$ of the above generator $\lambda_i$ in 
\[
R(T^m)=\mathbb{Z}[z_1, z_1^{-1}, \dots, z_m, z_m^{-1}],
\]
is the $i$-th elementary symmetric function of $z_1^{2}+z_1^{-2}, \dots, z_m^{2}+z_m^{-2}$.
The image $f_m^{*}(\Delta^{+})$, $f_m^{*}(\Delta^{-})$, $f_m^{*}(\Delta)$ in $R(T^{m})$ is given by
 \begin{align*}
f_m^{*}(\Delta^{+})&=\sum_{\varepsilon_1\cdots \varepsilon_m=+1} z_1^{\varepsilon_1}\dots z_m^{\varepsilon_m}, \\
f_m^{*}(\Delta^{-})&=\sum_{\varepsilon_1\cdots \varepsilon_m=-1} z_1^{\varepsilon_1}\dots z_m^{\varepsilon_m}, 
\end{align*}
and \[
f_m^{*}(\Delta)=\sum_{\varepsilon_1\cdots \varepsilon_m=\pm 1} z_1^{\varepsilon_1}\dots z_m^{\varepsilon_m}.
\]
Since $z^2_i+z_i^{-2}$ maps to $2$, for $1\leq i\leq m$, it is clear that 
\[
f_1^*(\lambda_i)=\alpha_i+\beta_i (z_1^2+z_1^{-2}),
\]
where $\alpha_i=2^i \displaystyle \binom{m-1}{i}$, $\displaystyle \beta_i=2^{i-1}\binom{m-1}{i-1}$.
It is also clear that
\begin{align*}
f_1^{*}(\Delta^{+})&=f_1^{*}(\Delta^{-})\\
&=2^{m-2}(z_1+z_1^{-1}), 
\\ 
f_1^{*}(\Delta)&=2^{m-1}(z_1+z_1^{-1}).
\end{align*}

Therefore, the total Chern classes are
\begin{align*}
c(f_1^{*}(\lambda_i))&=\{ (1+2u)(1-2u)\}^{\beta_i}\\&=(1-4u^2)^{\beta_i}, \\
c(f_1^*(\Delta^{+}))&=c(f_1^*(\Delta^{-}))\\
&=\{(1+u)(1-u)\}^{2^{m-2}}\\&=(1-u^2)^{2^{m-2}},\\
c(f_1^*(\Delta))&=\{(1+u)(1-u)\}^{2^{m-1}}\\&=(1-u^2)^{2^{m-1}}
\end{align*}
in $H^{*}(BT^1;\mathbb{Z})$.

Thus we have the following result on  the mod $2$ reduction of the total Chern classes:
\begin{proposition}\label{2}
The mod $2$ reduction of the total Chern classes of $f_1^{*}(\lambda_i)$, $f_1^{*}(\Delta^{\pm})$, $f_1^{*}(\Delta)$ are given by
\begin{align*}
c(f_1^{*}(\lambda_i))&=1, \\ c(f_1^{*}(\Delta^{+}))&=c(f_1^{*}(\Delta^{-}))\\ &=1+u^{2^{m-1}}, \\ c(f_1^{*}(\Delta))&=1+u^{2^{m}}
\end{align*}
in $H^{*}(BT^1;\mathbb{Z}/2)$.
In particular, $c(f_1^{*} \mu)=1+u^k$ where $\deg u^{k}=2k=2\dim \mu$ for $\mu=\Delta^{+}, \Delta^{-}, \Delta$.
\end{proposition}

\section{Generators}

In this section, let $H^{*}(X)$ be the mod $2$ cohomology $H^{*}(X;\mathbb{Z}/2)$ of $X$ and $\tilde{H}^{*}(X)$ the reduced mod $2$ cohomology $\tilde{H}^{*}(X;\mathbb{Z}/2)$. 
Recall that if $\xi$ is a real representation and if $\xi_{\mathbb{C}}$ is its complexification, then, 
the mod $2$ Chern class $c_{i}(\xi_{\mathbb{C}})$ is the square of the Stiefel-Whitney class $w_{i}(\xi)$, that is, $c_{i}(\xi_\mathbb{C})=w_i(\xi)^2$.

We recall Quillen's computation of the mod $2$ cohomology of $B\spin(n)$. As in the previous section, let $m=\left[\dfrac{n}{2}\right]$.
Depending on the type of the representation  of the spinor group $\spin(n)$, we  define 
$h$ to be $m-1$ or $m$ as follows:
\[
\begin{array}{c|c|c|c|c}
n & m & \mbox{type} & h  & h \\ \hline
8k & 4k & {\mathbb{R}} & 4k-1 &m-1 \\
8k+1 & 4k & {\mathbb{R}} & 4k-1 & m-1 \\
8k+2 & 4k+1 & {\mathbb{C}}  & 4k+1 & m \\
8k+3 & 4k+1 & {\mathbb{H}} & 4k+1 & m \\
8k+4 & 4k+2 & {\mathbb{H}} & 4k+2 &m \\
8k+5 & 4k+2 & {\mathbb{H}} & 4k+2  &m\\
8k+6 & 4k+3 &{\mathbb{C}} & 4k+3 &m \\
8k+7 & 4k+3 & {\mathbb{R}} & 4k+2 & m-1\\
\end{array}
\]
Let $p:\spin (n)\to SO(n)$ be the projection.
Then, the mod $2$ cohomology of $B\spin(n)$ is given by
$$H^{*}(B\spin(n))=H^{*}(BSO(n))/J \otimes \mathbb{Z}/2[z],$$
where $J$ is the ideal generated by $w_2, \sq^1 w_2, \dots, \sq^{2^{h-2}}\cdots\sq^1 w_2$ and $z$ is an element of degree $2^{h}$.
 Let $f_1:T^1\to \spin(n)$ be the inclusion of the first factor of the maximal torus.
Since $T^1$ is a closed subgroup of $\spin(n)$, the induced homomorphism $Bf_1^{*}:H^{*}(B\spin(n)) \to H^{*}(BT^1)$ is integral. On the other hand, since $c(f_1^{*}(\lambda_1))=1$, 
the induced homomorphism $(Bp\circ Bf_1)^*:\tilde{H}^{*}(BSO(n))\to \tilde{H}^{*}(BT^1)$ is zero.
Hence, $Bf_1^{*}(z)$ must be non-zero, that is, $Bf_1^{*}(z)=u^{2^{h-1}}$ and it generates the image of $Bf_1^*$.

Thus, we have the following proposition:
\begin{proposition}\label{3}
The image of $Bf_1^*$ is a polynomial ring generated by $u^{2^{h-1}}$, where $H^{*}(BT^1)=\mathbb{Z}/2[u]$ and $\deg u=2$.
\end{proposition}

Now, we consider the following representations for exceptional Lie groups $F_4$, $E_6$, $E_7$, $E_8$:
\[
\begin{array}{ccccccc} 
\mathrm{Spin}(9) & \stackrel{g_4}{\longrightarrow} & F_4 & \stackrel{\rho_4}{\longrightarrow} & SO(26) & {\longrightarrow} & SU(26), \vspace{1mm}\\ 
\mathrm{Spin}(10) & \stackrel{g_6}{\longrightarrow} & E_6 & \stackrel{\rho_6}{\longrightarrow} & SU(27),\vspace{1mm}\\ \mathrm{Spin}(12) & \stackrel{g_7}{\longrightarrow} & E_7 &\stackrel{\rho_7}{\longrightarrow} &  Sp(28) & \longrightarrow & SU(56),\vspace{1mm}\\ 
\mathrm{Spin}(16) & \stackrel{g_8}{\longrightarrow} & E_8 & \stackrel{\rho_8}{\longrightarrow} & SO(248) &
{\longrightarrow} & SU(248)
\end{array}
\]
such that
\begin{align*}
g_4^{*}(\rho_4)&=1+\lambda_1 + \Delta, \\
g_6^{*}(\rho_6)&=1+\lambda_1 + \Delta^+, \\
g_7^{*}(\rho_7)&=2\lambda_1 + \Delta^-, 
\\g_8^{*}(\rho_8)&=8+\lambda_2 + \Delta^+.
\end{align*}
The existence of such representations (and their construction) is proved in Adams' book \cite{adams-1996}.
These representations are tightly connected with the construction of  exceptional Lie groups.

With the following table, we summarize the information we need in the proof of Theorem~\ref{1}.
\[
\begin{array}{c|c|c|c|c|c}
G & \spin(n) & m & \mbox{rep.} & \mbox{$\dim$ of rep. }& \deg z = 2^h \\ \hline
F_4 & \spin(9) & 4 & \Delta & 16 &  16  \\
E_6 & \spin(10) & 5 & \Delta^{+} & 16  & 32 \\
E_7 & \spin(12) & 6 & \Delta^{-} & 32  & 64 \\
E_8 & \spin(16) & 8 & \Delta^{+} & 128  & 128
\end{array}
\]

\begin{proof}[Proof of Theorem~\ref{1} ]

For $G=E_6$, $E_7$, by Proposition~\ref{2}, we have 
$$c_{16}(f_1^{*}g_6^*\rho_6)=u^{16}, \quad c_{32}(f_1^{*}g_7^*\rho_7)=u^{32}.$$ 
By Proposition~\ref{3}, these elements are
indecomposable in $\mathrm{Im}\, Bf_1^*$. So, 
$c_{16}(\rho_6)$, $c_{32}(\rho_{7})$ are indecomposable in $H^{*}(BE_6)$, $H^*(BE_7)$.

For $G=F_4, E_8$, let $\rho_{4, \mathbb{C}}$, $\rho_{8,\mathbb{C}}$ be complexification of $\rho_4, \rho_8$, respectively. Then, by Proposition~\ref{2}, we have
$$c_{16}(f_1^{*}g_4^*\rho_{4,\mathbb{C}})=u^{16}, \quad c_{128}(f_1^{*}g_8^*\rho_{8,\mathbb{C}})=u^{128}.$$
These elements are decomposable in $\mathrm{Im}\,  Bf_1^*$.
However, since $\rho_{4,\mathbb{C}}$, $\rho_{8,\mathbb{C}}$ are complexification of $\rho_{4}$, $\rho_{8}$, we have $w_{16}(f_1^{*}g_4^*\rho_{4})=u^{8}$, $w_{128}(f_1^{*}g_8^{*}\rho_{8})=u^{64}$. By Proposition~\ref{3}, these elements are indecomposable in 
$\mathrm{Im}\,  Bf_1^*$. 
Hence, $w_{16}(\rho_{4})$, $w_{128}(\rho_{8})$ are indecomposable in $H^{*}(BF_4)$, $H^*(BE_8)$.
\end{proof}

\begin{remark}
If it is shown that the generators of the mod 2 cohomology of $BG$ is generated by two elements and one of these two generators is the degree $4$ element, say $y_4$, then
since $f_1^{*}(g_i^*(y_4))=0$, we have that $\mathrm{Im}\, f_1^*\circ g_i^*$ is generated by $f_1^*( g_4^*(w_{16}(\rho_4)))$, $f_1^*( g_6^*(c_{16}(\rho_6)))$, $f_1^*( g_7^*(c_{32}(\rho_7)))$, $f_1^*( g_8^*(w_{128}(\rho_8)))$, respectively. Thus, we do not need to refer the reader to Quillen's computation of the mod 2 cohomology of $B\spin(n)$ to complete the proof of Theorem~\ref{1}. This is the case for  $G=F_4, E_6, E_7$ and, if, for some $r$,  $E_r$-term of the spectral sequence converging to the mod $2$ cohomology of $BE_8$  
is also generated by two elements as an algebra over the Steenrod algebra, this argument is also applicable to the case $G=E_8$. 
\end{remark}

\end{document}